\documentclass{amsart}%
\usepackage{amsfonts}
\usepackage{amsmath}
\usepackage{amssymb}
\usepackage{graphicx}
\usepackage{hyperref}%
\providecommand{\U}[1]{\protect\rule{.1in}{.1in}}
\newtheorem{theorem}{Theorem}
\theoremstyle{plain}

\newtheorem{corollary}{Corollary}

\newtheorem{definition}{Definition}
\newtheorem{example}{Example}

\newtheorem{lemma}{Lemma}

\newtheorem{proposition}{Proposition}
\newtheorem{remark}{Remark}

\numberwithin{equation}{section}

\newcommand{\ann}{\operatorname{ann}}

\newcommand{\Z}{{\mathbb Z}}

\begin{document}
\title[{\normalsize On divisor topology of modules over integral domains}]{{\normalsize On divisor topology of modules over domains}}
\author{\"{U}nsal Tekir}
\address{Department of Mathematics, Marmara University, Istanbul, Turkey.}
\email{utekir@marmara.edu.tr}
\author{U\u{g}ur Yi\u{g}it }
\address{Department of Mathematics, Istanbul Medeniyet University, 34700,
\"{U}sk\"{u}dar, \.{I}stanbul, Turkey}
\email{ugur.yigit@medeniyet.edu.tr}
\author{Mesut Bu\u{g}day}
\address{Department of Mathematics, Marmara University, Istanbul, Turkey.}
\email{mesut.bugday@marmara.edu.tr}
\author{Suat Ko\c{c}}
\address{Department of Mathematics, Marmara University, Istanbul, Turkey.}
\email{suat.koc@marmara.edu.tr}

\begin{abstract}
Let $M\ $be a module over a domain $R$ and $M^{\#}=\{0\neq m\in M:Rm\neq M\}$
be the set of all nonzero nongenerators of $M.\ $Consider following
equivalence relation $\sim$ on $M^{\#}$ as follows: for every $m,n\in
M^{\#},\ m\sim n$ if and only if $Rm=Rn.\ $Let $EC(M^{\#})$ be the set of all
equivalence classes of $M^{\#}$ with respect to $\sim$.\ In this paper, we
construct a topology on $EC(M^{\#})$ which is called divisor topology of
$M\ $and denoted by $D(M).$\ Actually, $D(M)$ is extension of the divisor
topology $D(R)$ over domains in the sense of Yi\u{g}it and Ko\c{c} to modules. We
investigate separation axioms $T_{i}$ for every $0\leq i\leq5,$ first and
second countability, connectivity, compactness, nested property, and Noetherian
property on $D(M).\ $Also, we characterize some important classes of modules
such as uniserial modules, simple modules, vector spaces, and finitely
cogenerated modules in terms of $D(M).\ $Furthermore, we prove that $D(M)$ is
a Baire space for factorial modules. Finally, we introduce and study pseudo
simple modules which is a new generalization of simple modules, and use them to
determine when $D(M)$ is a discrete space.

\end{abstract}
\maketitle

\section{Introduction}

Studying new topologies on algebraic structures such as Zariski topology plays
a central role in finding possible connections between topological properties
and algebraic structures. For many years, many authors have created various
topologies on algebraic structures. The reason why Zariski topology is so
important is not only that it combines topology and algebra, but also that it
has profound roots in algebraic geometry. For more information about Zariski
topology and its generalizations, the reader may consult \cite{Callialp},
\cite{Ceken}, \cite{Ceken2}, and \cite{Yildiz}. In a recent study, Yi\u{g}it
and Ko\c{c} introduced divisor topology $D(R)\ $on an integral domain $R$ and
used it to give a topological proof of the infinitude of primes in an integral
domain (See, \cite[Theorem 8]{YiKo}). In this work, we introduce a
generalization of the divisor topology $D(R)$ to modules. Throughout the
paper, we \ focus only on integral domains and nonzero unital modules. Let $R$
will always denote such a domain and $M\ $will denote such an $R$-module. The
set of nonzero nongenerators of $M\ $is denoted by $M^{\#}=\{0\neq m\in
M:Rm\neq M\}.\ $We define a relation $\sim$ on $M^{\#}$ as follows: for every
$m,n\in M,$ $m\sim n$ if and only if $Rm=Rn$,$\ $or equivalently, $m\mid n$
and $n\mid m$ in $M$. Truly, $\sim$ is an equivalence relation on $M^{\#},$
and we denote the set of all equivalence classes of $M^{\#}$ with respect to
$\sim$ by $\operatorname{EC}(M^{\#})=\{[m]:m\in M^{\#}\},$ where $[m]$ denotes
the equivalence class of $m$ with respect to $\sim$. Consider the set
$U_{m}=\{[n]\in\operatorname{EC}(M^{\#}):n\mid m\}$ for each $m\in M^{\#}$,
and we prove that the family $\{U_{m}\}_{m\in M^{\#}}$ forms a basis for a
topology on $M^{\#}$ which is called the divisor topology on
$\operatorname{EC}(M^{\#})$. In the case of $M=R$ where $R$ is an integral domain
and $M={\mathbb{Z}}=R$, the divisor topology $D(M)$ corresponds to the divisor
topology on an integral domain $R$ and the divisor topology on integers
respectively, see \cite{Steen} and \cite{YiKo}. It is clear from the
construction of $D(M)$ that $D(M)$ is an extension of the divisor topology
over integral domains. In this work, we investigate not only possible
relations between topological properties of $D(M)$ and algebraic structures of
a module $M$ but also characterize modules arising from topological properties
of $D(M)$. Recall from \cite{FacSal} that an $R$-module $M$ is called a
\textit{uniserial module} if every two submodules of $M$ can be compared with
respect to the inclusion. Also, a topological space is called a \textit{nested
space} if open sets of this topology are linearly ordered. Among the other
results in Section 2, we prove in Lemma \ref{divide} that an $R$-module $M$ is
a uniserial module if and only if for any two nonzero elements $m,n\in M$, we
have either $m\mid n$ or $n\mid m$. We show afterward that $D(M)$ is a nested
space if and only if $M$ is a uniserial module (See, Theorem \ref{tnested}).
We prove that $D(M)\ $is always an Alexandrov space, and also it is a Baire space for
factorial modules (See, Proposition \ref{pAlex} and Theorem \ref{tBaire}).
Furthermore, we characterize simple modules and vector spaces by means of
Noetherian property of $D(M)$ (See, Theorem \ref{tNoetherian}).

In \cite[Proposition 2 and Proposition 5]{YiKo}, \ the authors showed that the
divisor topology $D(R)\ $is always a $T_{0}$-space but never a $T_{1}$ and Hausdorff space. Contrary to expectations, $D(M)$ may be a $T_{1}$-space (See,
Example \ref{ex1} and Example \ref{ex2}). In Section 3, we are closely
interested in the following question: When $D(M)\ $is a $T_{1}$-space? We give
an answer to the aforementioned question by means of pseudo simple modules in
the language of algebra (See, Theorem \ref{T_1-p.simple}). An $R$-module
$M\ $is said to be a \textit{pseudo simple module} if $Rm$ is simple for every
$m\in M^{\#}$. One can easily see that all simple modules are trivially pseudo
simple but the converse is not true in general. For instance, $%
\mathbb{Z}
$-module $%
\mathbb{Z}
_{6}$ is a pseudo simple module since $%
\mathbb{Z}
\overline{2}=%
\mathbb{Z}
\overline{4}$ and $%
\mathbb{Z}
\overline{3}$ are simple submodules while it is not a simple module. We
investigate the stability of pseudo simple modules under homomorphism, in
factor modules, in direct products and direct sums, in quotient modules, in the
trivial extension $R\ltimes M$ of an $R$-module $M\ $(See, Theorem \ref{homo},
Corollary \ref{cfac}, Theorem \ref{tdir}, Theorem \ref{tdir2}, Theorem
\ref{tquotient} and Theorem \ref{ttri}). Furthermore, we give a complete
classification for finitely generated pseudo simple ${\mathbb{Z}}$-modules
(See, Theorem \ref{fgPS}).

Section 4 is dedicated to the study of compactness, countability axioms, and
connectivity of $D(M)$. We note that the authors in \cite[Proposition
11]{YiKo} showed that $D(R)\ $can not be compact. Unlike $D(R),\ $it is possible that
$D(M)$ is a compact space (See, Example \ref{excom1}, Example \ref{excom2}
and Example \ref{excom3}). We prove that $D(M)\ $is a compact space if and
only if $M\ $has only finitely many simple submodules (See, Theorem
\ref{tcom}). Thanks to this result, we show that if $D(M)$ is a compact space,
then $M$ is a finitely cogenerated module (See, Theorem \ref{tfinitelycog}).
Also, we examine the first and second countability axioms for $D(M)\ $(See,
Proposition \ref{countable}). Moreover, we study the connectivity of the
divisor topology $D(M)\ $(See, Proposition \ref{ultraconnected} and Corollary
\ref{cconnected}).

Consequently, in Section 5, we study the further topological properties of
$D(M)$ such as Hausdorff, $T_{3}$-axiom, $T_{5}$-axiom, discrete and
(completely) normal properties (See, Theorem \ref{Hausdorff*condition},
Theorem \ref{tdiscrete}, Theorem \ref{main}, Proposition \ref{T5}, Corollary
\ref{ccompletely} and Remark \ref{rcompletely}).

\section{Fundamental Properties of Divisor Topology $D(M)$}

Let $M\ $be an $R$-module and $m,n\in M$. We say that an element $m^{\prime
}\in M$ is the greatest common divisor of $m$ and $n$, denoted by $m^{\prime
}=\gcd(m,n),$ if $m^{\prime}\mid m$ and $m^{\prime}\mid n$ with the property
that for any $m^{\prime\prime}\in M$ dividing both $n$ and $m$ implies
$m^{\prime\prime}\mid m^{\prime}$. In a similar fashion, one can define
$\operatorname{lcm}(m,n)$ for any $m,n\in M$. In the following, our first
result indicates that the family $\{U_{m}\}_{m\in M^{\#}}$ is a basis for a
topology on $\operatorname{EC}(M^{\#})$ which is called the divisor topology
of $M$, denoted by $D(M)$. We omit the proof since the proof is analogous to \cite[Proposition
$1$]{YiKo}.

\begin{proposition}
\label{Prop1} Let $M$ be an $R$-module. Then the following statements hold.

(i) $[m]\in U_{m}$ for any $m\in M^{\#}$.

(ii)$\ m\mid n$ if and only if $U_{m}\subseteq U_{n}$ for all $n,m\in M^{\#}$.

(iii) $\bigcup_{m\in M^{\#}}U_{m}=\operatorname{EC}(M^{\#})$

(iv)\ If $[x]\in U_{m}\cap U_{n}$ for some $m,n\in M^{\#},$ then $[x]\in
U_{x}\subseteq U_{m}\cap U_{n}$.

(v) If $\gcd(m,n)\in M^{\#}$ for some $n,m\in M^{\#},$ then $U_{m}\cap
U_{n}=U_{\operatorname{gcd}(n,m)}$.

(vi) Let $U_{m}\subseteq U_{n}$ and $U_{m^{\prime}}\subseteq U_{n}$ for some
$m,m^{\prime},n\in M^{\#}$.$\ $If $\operatorname{lcm}(m,m^{\prime})$ exists,
then $U_{\operatorname{lcm}(m,m^{\prime})}\subseteq U_{n}$.
\end{proposition}

\begin{remark}
We remark that if $M$ is chosen to be a simple module, then $\operatorname{EC}%
(M^{\#})=\emptyset$ and so $D(M)$ is an empty space. Therefore, we will assume
that $M$ is not a simple module unless otherwise specified.
\end{remark}

\begin{example}
Consider the ${\mathbb{Z}}$-module $M={\mathbb{Z}}$, it is clear that
${\mathbb{Z}}^{\#}=\{n\in{\mathbb{Z}}:n\neq0,\pm1\}$ and $U_{m}=\{[n]:m\text{
is integer multiple of $n$}\}$. In this case, $D(M)$ is actually the divisor
topology on the integers in the sense of Steen \cite{Steen}.
\end{example}

\begin{lemma}
\label{smallU_m} Let $M$ be an $R$-module and $m\in M^{\#}$. Then $U_{m}$ is
the smallest open set containing $[m]$.
\end{lemma}

\begin{proof}
The lemma holds in view of Proposition \ref{Prop1}.
\end{proof}

Given a topology, it is sensible to investigate that which separation axioms satisfies. A topological space $X$ is called a $T_{0}$\textit{-space} if, for
every two distinct points, there is an open set that contains exactly one of
them, not the other. $X$ is called a $T_{1}$\textit{-space} if, for every two
distinct points, there are open sets for each containing exactly one of them.
Also, $X$ is said to be a \textit{Hausdorff space} (or $T_{2}$\textit{-space})
if for every two distinct points there exist two disjoint open sets containing
them \cite{Munkres}. The authors in \cite[Proposition 4 and Proposition
5]{YiKo} showed that $D(R)\ $is always a $T_{0}$-space but not a $T_{1}$-space
and Hausdorff. One of the differences between $D(R)$ and $D(M)$ is that $D(M)$
may, in some cases, be a $T_{1}$ space or Hausdorff even discrete space. More
precisely, we give two examples; the first example shows that $D(M)$ is a
$T_{1}$-space and Hausdorff (indeed, $D(M)\ $is discrete), and the latter
example shows that it is not a $T_{1}$-space.

\begin{proposition}
\label{pT0}$D(M)$ is a $T_{0}$-space.
\end{proposition}

\begin{proof}
Choose two distinct points $[m]$ and $[n]$ in $\operatorname{EC}(M^{\#}),$
then $m\nmid n$ or $n\nmid m$. One may assume that $m\nmid n$, leading to
$[m]\not \in U_{n}$ and $[n]\in U_{n}$.
\end{proof}

\begin{example}
\label{ex1}Consider ${\mathbb{Z}}$-module $M={\mathbb{Z}}_{6}.$ Then
$\operatorname{EC}(M^{\#})=\{[\overline{2}],[\overline{3}]\}$ with the basis
$U_{\overline{2}}=\{[\overline{2}]\}$ and $U_{\overline{3}}=\{[\overline
{3}]\}$. In view of Proposition \ref{Prop1}, $D(M)$ is a discrete topology.
\end{example}

\begin{example}
\label{ex2}Consider ${\mathbb{Z}}$-module $M={\mathbb{Z}}_{12}.\ $Then
$\operatorname{EC}(M^{\#})=\{[\overline{2}]=[\overline{10}],[\overline
{3}]=[\overline{9}],[\overline{4}]=[\overline{8}],[\overline{6}]\}$ with the
basis $U_{\overline{2}}=\{[\overline{2}]\}$, $U_{\overline{3}}=\{[\overline
{3}]\}$, $U_{\overline{4}}=\{[\overline{2}],[\overline{4}]\}$ and
$U_{\overline{6}}=\{[\overline{2}],[\overline{3}],[\overline{6}]\}$. One can
see that $D(M)$ is not a $T_{1}$ space since $[\overline{2}]$ and
$[\overline{4}]$ can not be separated by two open sets.
\end{example}

Let $M$ be an $R$-module and $S$ be a subset of $M$. We say that an element
$m\in S$ an \textit{irreducible on }$S$ if whenever $m_{1}\mid m$ for some
$m_{1}\in S$ then $m\mid m_{1}$. It is easy to see that if we take
$S=M,\ $then irreducibles on $S$ are exactly irreducible elements of $M\ $in
the sense \cite{Lu}. For instance, consider ${\mathbb{Z}}$-module
${\mathbb{Z}}_{6}$ then the irreducibles of ${\mathbb{Z}}_{6}$ are
$\{\overline{1},\overline{5}\}$ and irreducibles on $S={\mathbb{Z}}_{6}^{\#}$
are $\{\overline{2},\overline{3},\overline{4}\}$.\ One can easily show that
irreducibles of a noncyclic $R$-module $M$ and irreducibles on $M^{\#}$
coincide. A torsion free $R$-module $M\ $is said to be a \textit{factorial
module} if $M\ $satisfies following two conditions: $(i)\ $every nonzero $m\in
M$ has irreducible factorization, that is,\ $m=q_{1}q_{2}\cdots q_{n}%
m^{\prime}$ for some irreducibles $q_{1},q_{2},\ldots,q_{n}\in R$ and some
irreducible $m^{\prime}\in M\ $and $(ii)$\ if $m=q_{1}q_{2}\cdots
q_{n}m^{\prime}=p_{1}p_{2}\cdots p_{t}m^{\prime\prime}$ has two irreducible
factorizations, then $n=t,\ q_{i}\sim p_{j}$ and $m^{\prime}\sim
m^{\prime\prime}$, that is, every irreducible factorization is unique up to
order and associated elements \cite{Lu}. A nonzero element $m\in M$ is said to
be a \textit{primitive} if $m\ |\ am^{\prime}$ for some $0\neq a\in R$ and
$m^{\prime}\in M,\ $then $m\ |\ m^{\prime}.\ $Every primitive element of a
torsion free module is irreducible, and the converse holds in factorial
modules (See, \cite[Theorem 2.1]{Lu}).

Recall that an element $x$ in the topological space $X$ is called an
\textit{isolated point} if the singleton $\{x\}$ is an open set in
$X\ $\cite{Munkres}. Now, we are ready to determine all isolated points of
$D(M)$ in terms of irreducibles on $M^{\#}.$

\begin{proposition}
\label{isolated} Let $M$ be an $R$-module (not necessarily a torsion free
module). Then $[m]\in\operatorname{EC}(M^{\#})$ is an isolated point if and
only if $m$ is an irreducible element on $M^{\#}$. In particular, if $M\ $is a
noncyclic module, then $[m]\in\operatorname{EC}(M^{\#})$ is an isolated point
if and only if $m$ is an irreducible element of $M.$ Furthermore, if $M\ $is a
factorial module, then $[m]\in\operatorname{EC}(M^{\#})$ is an isolated point
if and only if $m$ is a primitive element of $M.$
\end{proposition}

\begin{proof}
Suppose $[m]$ is an isolated point in $D(M),$ then there exists an open set
$O$ such that $\{[m]\}=O=\bigcup_{x\in\Lambda}U_{x}$ for some indexed set
$\Lambda\subseteq M^{\#}$. Since $U_{m}$ is the smallest open set containing
$[m]$, one must have $\{[m]\}=U_{m}$. Let $m^{\prime}\in M^{\#}$ be a divisor
of $m$ then $[m^{\prime}]\in U_{m}=\{[m]\}$ which implies that
$m\ |\ m^{\prime}.\ $Hence $m$ is irreducible on $M^{\#}$. Conversely, assume
that $m\ $is irreducible on $M^{\#}.\ $Let $[n]\in U_{m}$ then $n\mid m$ but
$m$ is an irreducible, so $m\ |\ n$ and hence $Rm=Rn$ if and only if
$[m]=[n]$. The rest follows from \cite[Theorem 2.1]{Lu}.
\end{proof}

Recall that a topological space $(X,\tau)$ is called a \textit{nested space} if $\tau$ is
linearly ordered by inclusion, or equivalently, every open set $O_{1},O_{2}$
in $X$ are comparable i.e. $O_{1}\subseteq O_{2}$ or $O_{2}\subseteq O_{1}%
\ $\cite{Richmand}. In \cite[Lemma $1$]{YiKo} the authors showed that a
topological space $(X,\tau)$ with a basis $\mathcal{B}$ is a nested space if
and only if any two elements of basis $\mathcal{B}$ is comparable. Now, we
will determine the conditions under which $D(M)$ is a nested topology. Before
this, we need the following result.

\begin{lemma}
\label{divide} Let $M$ be an $R$-module. Then, $M$ is an uniserial module if
and only if $m\mid n$ or $n\mid m$ for every $m,n\in M^{\#}$.
\end{lemma}

\begin{proof}
The first direction of the statement is straightforward. Conversely, suppose
that $m\mid n$ or $n\mid m$ for every $m,n\in M^{\#}$. Now, we will show that
$M\ $is uniserial. Assume that there are submodules $N$ and $L$ of $M$ such
that $N\not \subseteq L$ and $L\not \subseteq N$. It follows that there exist
elements $n\in N\setminus L$ and $l\in L\setminus N$. Then we have $n,l\in
M^{\#}$.\ By the assumption, $n\mid l$ or $l\mid n$, and thus we have
$Rn\subseteq Rl\subseteq L$ or $Rl\subseteq Rn\subseteq N,$ which both of them
are contradictions. Hence, $M$ is an uniserial module.
\end{proof}

\begin{theorem}\label{tnested}
Let $M$ be an $R$-module. Then, $M$ is an uniserial module if and only if $D(M)$ is a nested topology.
\end{theorem}

\begin{proof}
The proof holds in view of Lemma \ref{divide}. Suppose that $m\mid n$ or
$n\mid m$ for all $m,n\in M^{\#}$. It follows that either $U_{m}\subseteq
U_{n}$ or $U_{n}\subseteq U_{m}$, which leads to the result. Conversely, choose $m,n\in
M-\{0\}.$ It is clear that if at least one of them is a generator, then the
proof is done. Now, assume that they are not generators, that is, $m,n\in
M^{\#}$. Then, $U_{m}\subseteq U_{n}$ or $U_{n}\subseteq U_{m}$ by the
assumption. This implies that $m\ |\ n\ $or $n\ |\ m,\ $and so $M$ is a
uniserial module.
\end{proof}

A topological space is called an \textit{Alexandrov space} if the intersection
of an arbitrary family of open sets is also open \cite{Arenas}. This is
equivalent to the existence of a minimal neighborhood for every point $x$ of
$X$.

\begin{proposition}
\label{pAlex}$D(M)$ is an Alexandrov space.
\end{proposition}

\begin{proof}
The proof follows directly from Lemma \ref{smallU_m}.
\end{proof}

\begin{proposition}
\label{closure} Let $M$ be an $R$-module with $[m]\in\operatorname{EC}%
(M^{\#})$. Then,
\[
\overline{\{[m]\}}=\{[n]\in\operatorname{EC}(M^{\#}):m\mid n\}.
\]

\end{proposition}

\begin{proof}
Let $[n]\in\overline{\{[m]\}}$ for some $n\in M^{\#}$. It follows that any
open set containing $[n]$ necessarily contains $[m]$. Since $\left[  n\right]
\in U_{n}$, we have $\left[  m\right]  \in U_{n},\ $that is, $m\ |\ n.\ $%
Conversely, suppose $m\mid n$ for some $n\in M^{\#}$ and choose an open set
$O$ such that $O\cap\{[n]\}\neq\emptyset$. It is obvious that $[m]\in
U_{m}\subseteq U_{n}\subseteq O$ by Lemma \ref{smallU_m}. Thus, $\left[
n\right]  \in\overline{\{[m]\}}$.
\end{proof}

Let $M$ be an $R$-module. The set of all irreducible elements of $M$ is
denoted by $\operatorname{Irr}(M).\ $It is clear that if $M\ $is a cyclic
module, then $\operatorname{Irr}(M)=u(M).\ $ where $u(M)$ is the set of all generators of $M$. If $M\ $is a noncyclic module,
then $\operatorname{Irr}(M)\cap u(M)=\emptyset.$ Recall from \cite{Munkres}
that a topological space $X\ $is said to be a \textit{Baire space} if for
every countable family $\left\{  F_{n}\right\}  _{n\in%
\mathbb{N}
}$ of closed sets having an empty interior, then $%
{\displaystyle\bigcup\limits_{n\in\mathbb{N}}}
F_{n}\ $also has an empty interior, or equivalently, for every countable family
$\left\{  O_{n}\right\}  _{n\in%
\mathbb{N}
}$ of open dense sets, their intersection $%
{\displaystyle\bigcap\limits_{n\in\mathbb{N}}}
O_{n}\ $is dense in $X\ $(See, \cite[Lemma 48.1]{Munkres}). In the following
result, we investigate the open dense set in $D(M)$ for any factorial module
$M,$ and use this fact to prove $D(M)$ is a Baire space for factorial modules.

\begin{theorem}
\label{dense}Let $M$ be a factorial $R$-module with the property that
$\operatorname{Irr}(M)\cap u(M)=\emptyset$. Then, $\operatorname{EC}%
(\operatorname{Irr}(M))$ is dense in $D(M)$. Furthermore, each open dense set
in $D(M)$ contains $\operatorname{EC}(\operatorname{Irr}(M))$.
\end{theorem}

\begin{proof}
We only need to show that $\operatorname{EC}(M^{\#})\subseteq\overline
{\operatorname{EC}(\operatorname{Irr}(M))}$. Any nonzero element $m$ of $M$
has an irreducible factorization $m=r_{1}r_{2}\dots r_{n}m^{\prime}$ where
$r_{1},\dots,r_{n}$ are irreducibles in $R$ and $m^{\prime}$ is irreducible in
$M$. Then, $[m]\in\{\overline{[m^{\prime}]}\}\subseteq\overline
{\operatorname{EC}(\operatorname{Irr}(M))}$ by Proposition \ref{closure} which
implies that $\operatorname{EC}(M^{\#})=\overline{\operatorname{EC}%
(\operatorname{Irr}(M))}$. In addition, take an open dense set $O$ in $D(M)$
then $\overline{O}=\operatorname{EC}(M^{\#})$. Now, let $m\in
\operatorname{Irr}(M)$ then $\left[  m\right]  \in\overline{\operatorname{EC}%
(\operatorname{Irr}(M))}=\operatorname{EC}(M^{\#})=\overline{O}$. It follows
that for any open set $O^{\prime}$ containing $[m]$ intersects with $O$. In
particular, one may choose $O^{\prime}=U_{m}$, and in this case $O^{\prime
}=U_{m}=\{[m]\}$ by Proposition \ref{isolated}. Since $U_{m}\cap
O\neq\emptyset$ we get $[m]\in O$. Therefore, $\operatorname{EC}%
(\operatorname{Irr}(M))\subseteq O$.
\end{proof}

\begin{theorem}
\label{tBaire}Let $M$ be a factorial $R$-module with the property that
$\operatorname{Irr}(M)\cap u(M)=\emptyset$. Then $D(M)$ is a Baire space.
\end{theorem}

\begin{proof}
Follows from Theorem \ref{dense} and \cite[Lemma 48.1]{Munkres}.
\end{proof}

\bigskip The concept of Noetherian rings has a distinguished place in
commutative algebra and algebraic geometry. A ring $R$ is said to be a
\textit{Noetherian ring} if the ascending chain condition holds for an
arbitrary family of its ideals, or equivalently, each ideal of $R$ is finitely
generated \cite{Sharp}. Similarly, recall from \cite{Hars} that a topological
space $X$ is called a \textit{Noetherian space} if the descending chain
condition holds for closed subsets of $X$, alternatively, an ascending chain
condition holds for open subsets of $X$. In the literature, there is a strong
connection between Noetherian rings and Noetherian space in terms of Zariski
topology on $Spec(R):$ if $R\ $is a Noetherian ring, then the Zariski topology
on $Spec(R)$ is a Noetherian space, but the converse is not true in general
\cite{Ohm}. The reader may consult \cite{Hamed}, \cite{Ohm} and
\cite{OzNaTeKo} for more details on rings satisfying the converse of the statement.

Let $X$ be a topological space with the basis $\mathfrak{B}$.\ The authors in
\cite[Theorem 6]{YiKo} showed that $X$ is a Noetherian space if and only if
$X$ satisfies the minimum condition on closed sets if and only if $X\ $satisfies
the maximum condition on open sets if and only if any subfamily $\mathfrak{B}%
^{\prime}$ of $\mathfrak{B}$ has a maximal element. Now, we are ready to
determine the conditions under which $D(M)$ is a Noetherian space.

\begin{theorem}
\label{tNoetherian}Let $M$ be a torsion free $R$-module. Then $D(M)$ is a
Noetherian space if and only if either $M\ $is a simple module or $M$ is a
vector space over the field $R.\ $
\end{theorem}

\begin{proof}
$\left(  \Rightarrow\right)  :\ $Suppose that $D(M)$ is a Noetherian space.
Then, every subfamily $\mathfrak{B}^{\prime}$ of the basis $\mathfrak{B}%
=\left\{  U_{m}\right\}  _{m\in M^{\#}}$ has a maximal element by \cite[Theorem
6]{YiKo}. If there is no element $m\in M^{\#},\ $then $M$ is simple. Now,
assume that $M\ $is not simple so one may choose $m\in M^{\#}.\ $Now, we will
show that $R\ $is a field. Take an element $0\neq a\in R.\ $Since $M\ $is
torsion free, we have $0\neq a^{k}m\in M^{\#}$ and $a^{k}m\ |\ a^{k+1}m$ for
all $k\in%
\mathbb{N}
.\ $Now, put $\mathfrak{B}^{\prime}=\left\{  U_{a^{k}m}\right\}  _{k\in%
\mathbb{N}
}$ and then $\mathfrak{B}^{\prime}$ has a maximal element, say $U_{a^{t}m}%
.\ $Since $U_{a^{t}m}\subseteq U_{a^{t+1}m},\ $we conclude that $U_{a^{t}%
m}=U_{a^{t+1}m}$ which yields that $a^{t}m=ra^{t+1}m$ for some $r\in R.\ $As
$M\ $is a torsion free module, we have $ra=1$, that is, $a$ is a unit. Hence,
$R\ $is a field.

$\left(  \Leftarrow\right)  :$\ Suppose first $M\ $is a simple module then
$D(M)$ is an empty space, so it is trivially a Noetherian space. Now, assume that
$M$ is a vector space over the field $R.\ $Let $\mathfrak{B}^{\prime}$ be a
subfamily of the basis $\mathfrak{B}=\left\{  U_{m}\right\}  _{m\in M^{\#}}%
.$\ Let $U_{m_{1}}\subseteq U_{m_{2}}$ for some $U_{m_{1}},U_{m_{2}}%
\in\mathfrak{B}^{\prime}.$\ Then $m_{1}\ |\ m_{2}$ so $m_{2}=am_{1}$ for some
$0\neq a\in R.\ $Since $R$ is a field, $m_{1}=a^{-1}m_{2}$ which implies that
$m_{2}\ |\ m_{1}.\ $Thus, we have $U_{m_{1}}=U_{m_{2}}$ which means that every
nonempty subcollection $\mathfrak{B}^{\prime}$ of $\mathfrak{B}$ has a maximal element. Then by
\cite[Theorem 6 (iv)]{YiKo}, $D(M)$ is a Noetherian space.
\end{proof}

\section{Pseudo simple modules vs $T_1$-space}

\bigskip From Example \ref{ex1} and Example \ref{ex2}, one may ask when
$D(M)\ $is a $T_{1}$-space? In this section, we introduce pseudo simple
modules, which are a new generalization of simple modules. We answer the
aforementioned question in the language of algebra by means of pseudo simple
modules. In this section, $R\ $is assumed to be a commutative ring with a
nonzero identity (not necessarily a domain).

\begin{definition}
An $R$-module $M$ is called a pseudo simple if $Rm$ is simple for each $0\neq
m\in M$ with $Rm\neq M$.\ In particular, a ring $R\ $is said to be a pseudo
simple ring if it is a pseudo simple module over itself.
\end{definition}

\begin{lemma}
\label{LemmaPseu} An $R$-module $M$ is a pseudo simple if and only if
$\operatorname{ann}(m)$ is a maximal ideal for every $0\neq m\in M$ with
$Rm\neq M$.
\end{lemma}

\begin{proof}
Suppose that $M\ $is a pseudo simple module and choose $0\neq m\in M$ with
$Rm\neq M$. Let $a\not \in \operatorname{ann}(m)$ then $Ram$ is a nonzero
submodule of $Rm$. It follows that $Ram=Rm$ and then $m=ram$ for some $r\in
R$. Hence, $1-ra\in\operatorname{ann}(m)$, that is, $\operatorname{ann}(m)$ is
a maximal ideal of $R$. Conversely, suppose that $\operatorname{ann}(m)$ is a
maximal ideal for each $0\neq m\in M$ with $Rm\neq M$. Now, take a nonzero
cyclic submodule of $Rm$, namely $Ram=N\subseteq Rm$. Since $N$ is a nonzero
submodule then $a\not \in \operatorname{ann}(m)$. We can now conclude that
$1-ra\in\operatorname{ann}(m)$ by the maximality of $\operatorname{ann}(m)$.
This gives $N=Ram=Rm,\ $and thus $Rm\ $is simple.
\end{proof}

We will list down examples with a non-example of pseudo simple modules.

\begin{example}
(i) Every vector space $V$ is a pseudo simple module

(ii)\ Every simple module is trivially a pseudo simple module.

(iii)$\ {\mathbb{Z}}_{2}\times{\mathbb{Z}}_{2}$ is a pseudo simple as
${\mathbb{Z}}$-module.

(iv) ${\mathbb{Z}}$-module ${\mathbb{Z}}_{6}$ is pseudo simple module apart
from the examples $i)$ and $ii)$. Let $\overline{0}\neq\overline{m}%
\in{\mathbb{Z}}_{6}$ be a nonzero element such that ${\mathbb{Z}}\overline
{m}\neq{\mathbb{Z}}_{6}$. Then, $\overline{m}$ must be in the set
$\{\overline{2},\overline{3},\overline{4}\}$. Observing that
$\operatorname{ann}(\overline{2})=\operatorname{ann}(\overline{4}%
)=3{\mathbb{Z}}$ and $\operatorname{ann}(\overline{3})=2{\mathbb{Z}}$ are both
maximal ideals of ${\mathbb{Z}}$. Therefore, ${\mathbb{Z}}$-module
${\mathbb{Z}}_{6}$ is a pseudo simple ${\mathbb{Z}}$-module.

(v)$\ {\mathbb{Z}}$-module ${\mathbb{Z}}_{12}$ is not a pseudo simple module
since $\operatorname{ann}(\overline{3})=4{\mathbb{Z}}$ is not a maximal ideal.
\end{example}

In the following discussion, we are interested in the following problem. Given
a finitely generated abelian group $G$, is there a way to decide whether $G$
is a pseudo simple ${\mathbb{Z}}$-module or not? The following theorem
provides an affirmative answer.

\begin{theorem}
\label{fgPS} A finitely generated abelian group $G$ is a pseudo simple
${\mathbb{Z}}$-module if and only if $G\cong{\mathbb{Z}}_{p}\oplus{\mathbb{Z}%
}_{q}$ for some distinct primes $p$ and $q$ or $({\mathbb{Z}}_{p})^{n}$ where
$n\geq1$.
\end{theorem}

\begin{proof}
The fundamental theorem of finitely generated abelian groups asserts that any
finitely generated abelian group $G$ can be expressed as
\[
G\cong{\mathbb{Z}}^{n}\oplus{\mathbb{Z}}_{q_{1}}\oplus{\mathbb{Z}}_{q_{2}%
}\oplus\dots\oplus{\mathbb{Z}}_{q_{r}},\quad n\geq0
\]
where $q_{1},\dots,q_{t}$ are prime powers i.e., $q_{r}=p_{r}^{m_{r}}$ for
some prime numbers $p_{r}$ and $m_{r}\geq1$.

Assume $n\geq1$ and let $m=(1,1,\dots,1,\overline{0},\dots,\overline{0})\in G$
then $\operatorname{ann}(m)=0$ which is not maximal, so $G$ is not a pseudo
simple module by Lemma \ref{LemmaPseu}. Therefore, $n=0$ if $G$ is a finitely
generated abelian pseudo simple ${\mathbb{Z}}$-module.

Now, assume that $G\cong{\mathbb{Z}}_{p_{1}^{m_{1}}}\oplus{\mathbb{Z}}%
_{p_{2}^{m_{2}}}\oplus\dots\oplus{\mathbb{Z}}_{p_{r}^{m_{r}}}$ for some prime
numbers $p_{i}$ and $m_{i}\geq1$. Moreover, assume that $m_{1}\geq2$ and set
$m=(\overline{1},\overline{0},\dots,\overline{0})\in G$. This gives
$\operatorname{ann}(m)=p_{1}^{m_{1}}{\mathbb{Z}}$ which is not maximal since
$m_{1}\geq2$. Thus, $m_{i}=1$ for each $i$ if $G$ is pseudo simple.

Now assume that $G\cong{\mathbb{Z}}_{p_{1}}\oplus{\mathbb{Z}}_{p_{2}}%
\oplus\dots\oplus{\mathbb{Z}}_{p_{r}}$ for some prime numbers $p_{1}%
,\dots,p_{r}$. It is obvious that if $r=1$ then $G$ is simple and so pseudo
simple. If $r=2$ then $G\cong{\mathbb{Z}}_{p_{1}}\oplus{\mathbb{Z}}_{p_{2}}$.
In case of $p_{1}=p_{2}$, $\operatorname{ann}(m)=p_{1}{\mathbb{Z}}$ is maximal
for every $0\neq m\in G$. Thus, $G$ is a pseudo simple ${\mathbb{Z}}$-module.
If $p_{1}\neq p_{2}$, consider the nonzero non-generator element
$(\overline{0},\overline{0})\neq(\overline{x},\overline{y})=m^{\prime}\in G$
i.e., $\overline{x}=0$ and $\overline{y}\neq\overline{0}$ or $\overline{x}%
\neq0$ and $\overline{y}=\overline{0}$. It follows that $\operatorname{ann}%
(m^{\prime})=p_{2}{\mathbb{Z}}$ or $\operatorname{ann}(m^{\prime}%
)=p_{1}{\mathbb{Z}}$ which both of them are maximal. Therefore $G$ is a pseudo simple.

Lastly, assume $r\geq3$ and we will show that $\{p_{1},\dots,p_{r}\}$ contains
exactly one prime, that is, $p_{1}=\dots=p_{r}$. Assume $p_{1}\neq p_{2}$ and
set $m=(\overline{1},\overline{1},\overline{0},\dots,\overline{0})\in G$ then
$\operatorname{ann}(m)=p_{1}p_{2}{\mathbb{Z}}$ which is not maximal which
completes the proof.
\end{proof}

\begin{corollary}
\label{pseudoZn}${\mathbb{Z}}$-module ${\mathbb{Z}}_{n}$ is pseudo simple if
and only if $n=p,p^{2}$ or $pq$ where $p$ and $q$ are distinct primes.
\end{corollary}

\begin{proof}
The statement holds in view of Theorem \ref{fgPS}. Let $n=p_{1}^{\alpha_{1}%
}\dots p_{r}^{\alpha_{r}}$ for some distinct primes and $\alpha_{i}\geq1$.
Remark that $r\leq2$ by Theorem \ref{fgPS}. If $r=2$ then $\alpha_{1}%
=\alpha_{2}=1$, i.e., $n=p_{1}p_{2}$ for some distinct primes. Now, assume
that $n=p^{\alpha}$ for some $\alpha\geq1$. It is clear that, ${\mathbb{Z}}%
$-module ${\mathbb{Z}}_{n}$ is pseudo simple if $\alpha=1$. In case of
$\alpha=2$, choose $\overline{0}\neq\overline{m}\in{\mathbb{Z}}_{p^{2}}$ which
is not a generator then $m=pk$ for some $k\in{\mathbb{Z}}$ with $\gcd(k,p)=1$.
It follows that $\operatorname{ann}(\overline{m})=\operatorname{ann}%
(\overline{pk})=\operatorname{ann}(\overline{p})=p{\mathbb{Z}}$. Therefore,
${\mathbb{Z}}$-module ${\mathbb{Z}}_{p^{2}}$ is a pseudo simple.

Now, assume $\alpha\geq3$ then $\operatorname{ann}(\overline{p})=p^{\alpha
-1}{\mathbb{Z}}$ is not maximal since $\alpha\geq3$. Thus, ${\mathbb{Z}}%
$-module ${\mathbb{Z}}_{p^{\alpha}}$ is not pseudo simple if $\alpha\geq3$.
\end{proof}

\begin{theorem}
\label{homo} Let $\phi:M\rightarrow N$ be an $R$-module homomorphism. Then the
following statements hold.

(i) If $M$ is a pseudo simple $R$-module and $\phi$ is a surjective, then $N$
is a pseudo simple $R$-module.

(ii)\ If $N$ is a pseudo simple $R$-module and $\phi$ is an injective, then
$M$ is a pseudo simple $R$-module.
\end{theorem}

\begin{proof}
$(i):$\ Let $0\neq n\in N$ be a nongenerator element such that $\phi(m)=n$ for
some $0\neq m\in M$. Note that $m$ cannot be a generator and
$\operatorname{ann}(m)\subseteq\operatorname{ann}(\phi(m))=\operatorname{ann}%
(n)$. Hence, $\operatorname{ann}(n)$ is a maximal ideal.

$(ii):$\ Let $0\neq m\in M$ be a nongenerator element then $\phi(m)$ can not
be a generator as well. It follows that $\operatorname{ann}%
(m)=\operatorname{ann}(\phi(m))$. Thus, $M$ is a pseudo simple module.
\end{proof}

The following results are immediate consequences of Theorem \ref{homo}.

\begin{corollary}
\label{cfac}(i)\ If $M$ is a pseudo simple $R$-module then each submodule of
$M$ is a pseudo simple $R$-module.

(ii)\ If $M$ is a pseudo simple module then each factor module of $M$ by a
submodule is a pseudo simple $R$-module.

(iii)\ Let $\{M_{i}\}_{i\in I}$ be a set of all $R$-modules $M_{i}$ such that
$\bigoplus_{i}M_{i}$ or $\prod_{i}M_{i}$ is a pseudo simple module. Then, each
$M_{i}$ is a pseudo simple module $R$-module.
\end{corollary}

\begin{theorem}
Let $M$ be an $R$-module with a submodule $N$. Then, $M/N$ is a pseudo simple
$R$-module if and only if $(N:m)$ is a maximal ideal for each $0\neq m\in M$
with $R(m+N)\neq M/N$.
\end{theorem}

\begin{proof}
The first direction follows from the fact that $\operatorname{ann}%
(m+N)=(N:m)$. Conversely, let $\operatorname{ann}(m+N)$ be a maximal ideal for
each $0\neq m\in M$ with $R(m+N)\neq M/N$. Take a nonzero cyclic submodule of
$R(m+N)$, say $M^{\prime}=R(am+N)\subseteq R(m+N)$. Since $M^{\prime}$ is
nonzero then $am\not \in N$ so $a\not \in (N:m)$. Hence, $1-ra\in(N:m)$ for
some $r\in R$. Moreover, $m-ram\in N$. Thus, $M^{\prime}=R(am+N)=R(m+N)$. That
is, $R(m+N)$ is simple.
\end{proof}

\begin{theorem}
\label{tquotient}Suppose $M$ is an $R$-module with a multiplicatively closed
subset $S$ of $R$. If $M$ is a pseudo simple $R$-module with
$\operatorname{ann}(m)\cap S=\emptyset$ for all $m\in M^{\#}$ then $S^{-1}M$
is a pseudo simple $S^{-1}R$ module.
\end{theorem}

\begin{proof}
Let $\frac{m}{s}$ be a nonzero and nongenerator element of $S^{-1}M$. Note
that $\operatorname{ann}(m)\cap S=\emptyset$ and $0\neq m\in M$ is
nongenerator as well. Since $M$ is a pseudo simple then $\operatorname{ann}%
(m)$ is a maximal ideal for each nonzero nongenerator element $m$ of $M$.
Moreover, $S^{-1}(\operatorname{ann}_{R}(m))=\operatorname{ann}_{S^{-1}%
R}(\frac{m}{s})$ is a maximal ideal. Hence, $S^{-1}M$ is a pseudo simple.
\end{proof}

\begin{theorem}
\label{tdir}Let $M_{1}$ and $M_{2}$ be $R$-modules such that at least one of
them is a noncyclic module. Then $M=M_{1}\oplus M_{2}$ is a pseudo simple
$R$-module if and only if $M_{1}$ and $M_{2}$ are pseudo simple $R$-modules
and $\operatorname{ann}(M_{1})=\operatorname{ann}(M_{2})$ is a maximal ideal
of $R$.
\end{theorem}

\begin{proof}
$\Rightarrow:$ It remains to show that $\operatorname{ann}(M_{1}%
)=\operatorname{ann}(M_{2})$ is a maximal ideal of $R$ by Corollary \ref{cfac} (iii).
Let $0\neq m\in M_{1}$ be a fixed element and choose an arbitrary element
$0\neq m^{\prime}\in M_{2}$, then $(0,0)\neq(m,m^{\prime})$ is not a generator
of $M$. It follows that $\operatorname{ann}((m,m^{\prime}))=\operatorname{ann}%
(m)\cap\operatorname{ann}(m^{\prime})$ is a maximal ideal of $R$. Hence we
have $\operatorname{ann}(m)=\operatorname{ann}(m)\cap\operatorname{ann}%
(m^{\prime})$ as $\operatorname{ann}(m)\cap\operatorname{ann}(m^{\prime
})\subseteq\operatorname{ann}(m)$, and this shows that $\operatorname{ann}%
(m)=\operatorname{ann}(m^{\prime})$ is a maximal ideal of $R$. Moreover, we
conclude that $\operatorname{ann}(m)=\cap_{m^{\prime}\in M_{2}}%
\operatorname{ann}(m^{\prime})=\operatorname{ann}(M_{2})$ since $m^{\prime}\in
M_{2}$ is arbitrary. By replacing $m\in M_{1}$ with arbitrary element
$m^{\star}\in M_{1},\ $we obtain $\operatorname{ann}(M_{1})=\operatorname{ann}%
(M_{2}).\ $

$\Leftarrow:$ Suppose that $\ann(M_1)=\ann(M_2)$ is a maximal ideal. Then $\operatorname{ann}(M)=\operatorname{ann}%
(M_{1})\cap\operatorname{ann}(M_{2})=\operatorname{ann}(M_{1})$ is a maximal
ideal. This gives $M=M_{1}\oplus M_{2}$ is a pseudo simple module.
\end{proof}

\begin{theorem}
\label{tdir2}Let $M_{1}$ and $M_{2}$ be two cyclic $R$-modules and
$M=M_{1}\oplus M_{2}$. Then $M$ is pseudo simple $R$-module if and only if
$M_{1}$ and $M_{2}$ are simple modules and one of the following conditions hold:

(i)$\ \operatorname{ann}(M_{1})\neq\operatorname{ann}(M_{2})$ are maximal
ideals of $R$ and $(m_{1},m_{2})$ is a generator of $M$ for all nonzero
$m_{1}$ and $m_{2}$.

(ii)$\ \operatorname{ann}(M_{1})=\operatorname{ann}(M_{2})$ is a maximal ideal
of $R$.
\end{theorem}

\begin{proof}
Suppose $M$ is a pseudo simple $R$-module with $M_{1}=Rx$ and $M_{2}=Ry$ for
some $x\in M_{1}$ and $y\in M_{2}$. We will show that $M_{1}$ is a simple
$R$-module. Since $(x,0)$ is non-generator of $M$ and $M$ is a pseudo simple,
we have $\operatorname{ann}(x,0)=\operatorname{ann}(x)=\operatorname{ann}%
(M_{1})$ is a maximal ideal. This gives $M_{1}$ is a simple module as
$M_{1}\cong R/\operatorname{ann}(x)$. Similarly, one can show that $M_{2}$ is
a simple module. Assume that $\operatorname{ann}(M_{1})\neq\operatorname{ann}%
(M_{2})$. Now, we need to show that $(m_{1},m_{2})$ is a generator of $M$ for
all nonzero $m_{1}$ and $m_{2}$. Suppose to the contrary that $(m_{1},m_{2})$
is not a generator of $M$ for some nonzero $m_{1}$ and $m_{2}$. Since $M$ is a
pseudo simple, $\operatorname{ann}((m_{1},m_{2}))=\operatorname{ann}%
(m_{1})\cap\operatorname{ann}(m_{2})=\operatorname{ann}(M_{1})\cap
\operatorname{ann}(M_{2})$ is a maximal ideal. Then, $\operatorname{ann}%
(M_{1})\cap\operatorname{ann}(M_{2})\subseteq\operatorname{ann}(M_{1}%
),\operatorname{ann}(M_{2})$. Hence, $\operatorname{ann}(M_{1}%
)=\operatorname{ann}(M_{2})$ as $\operatorname{ann}(M_{1})$ is maximal ideal
which completes the proof.

For the only if part, if $(ii)$ holds, then note that $\operatorname{ann}%
(M)=\operatorname{ann}(M_{1})\cap\operatorname{ann}(M_{2})=\operatorname{ann}%
(M_{1})$ is a maximal ideal of $R$. This ensures that $M$ is a pseudo simple
module. Now, assume that $M_{1}$ and $M_{2}$ are two simple modules with
different maximal ideals $\operatorname{ann}(M_{1})$ and $\operatorname{ann}%
(M_{2})$ and also $(m_{1},m_{2})$ is a generator of $M$ for all nonzero
$m_{1}$ and $m_{2}$. Assume moreover that $(m,m^{\prime})\in M$ is a nonzero
nongenerator then by assumption we have either $m=0$ or $m^{\prime}=0$. One
may assume that $m^{\prime}=0$. It follows that $\operatorname{ann}%
((m,m^{\prime}))=\operatorname{ann}(m)=\operatorname{ann}(M_{1})$ is a maximal
ideal. Hence $M$ is a pseudo simple $R$-module since $M_{1}$ is a simple
module and $m\neq0$.
\end{proof}

Let $M$ be an $R$-module. The trivial extesion $R\ltimes M=R\oplus M=\left\{
(r,m):r\in R,m\in M\right\}  $ is a commutative ring with componentwise
addition and the multiplication defined by $(a,m)(b,n)=(ab,an+bm)$ for every
$a,b\in R$ and $m,n\in M$\ \cite{Anderson}.

\begin{theorem}
\label{ttri}Let $M$ be an $R$-module. Then $R\ltimes M$ is a pseudo simple ring
if and only if $R$ is a local ring with a unique maximal ideal
$\operatorname{ann}(M)$ and $\operatorname{ann}(r)=\operatorname{ann}(M)$ for
all nonzero nonunit elements $r$ of $R$.
\end{theorem}

\begin{proof}
$\Rightarrow:$ Let $R\ltimes M$ be a pseudo simple ring. We may assume that
$R$ is not a field. Now, take a nonzero nonunit element $a$ of $R$, then
$(0,0)\neq(a,0)$ is a nongenerator of $R\ltimes M$. Since $R\ltimes M$ is a
pseudo simple ring, we have $\operatorname{ann}((a,0))=\operatorname{ann}%
_{R}(a)\ltimes\operatorname{ann}_{M}(a)$ is a maximal ideal of $R\ltimes M$.
Then by \cite[Theorem 3.2]{Anderson}, $\operatorname{ann}_{M}(a)=M$ and
$\operatorname{ann}_{R}(a)$ is a maximal ideal of $R$. This gives
$a\in\operatorname{ann}(M)$ which implies that $\operatorname{ann}(M)$ is the
unique maximal ideal of $R$. Also note that $\operatorname{ann}%
(a)=\operatorname{ann}(M)$.

$\Leftarrow:$ Let $(0,0)\neq(a,m)$ be a nongenerator of $R\ltimes M$. Assume
that $a=0$ then this gives $\operatorname{ann}((0,m))=\operatorname{ann}%
(m)\ltimes M=\operatorname{ann}(M)\ltimes M$ since $\operatorname{ann}(M)$ is
the unique maximal ideal and $\operatorname{ann}(M)\subseteq\operatorname{ann}%
(m)$. Then by \cite[Theorem 3.2]{Anderson}, $\operatorname{ann}%
((0,m))=\operatorname{ann}(M)\ltimes M$ is a maximal ideal. Now, we assume
that $a\neq0$ then this gives $a$ is a nonunit since $(a,m)$ is a nongenerator
element of $R\ltimes M$. Then, $a\in\operatorname{ann}(M)$ which implies that
$\operatorname{ann}_{M}(a)=M$. It follows from the assumption that
$\operatorname{ann}((a,m))=\operatorname{ann}(a)\cap\operatorname{ann}%
(m)\ltimes\operatorname{ann}_{M}(a)=\operatorname{ann}(M)\ltimes M$. Thus,
$\operatorname{ann}((a,m))=\operatorname{ann}(M)\ltimes M$ is a maximal ideal
of $R\ltimes M$.
\end{proof}

\begin{theorem}
\label{T_1-p.simple} Let $M\ $be a module over an integral domain $R.\ $Then
$D(M)$ is a $T_{1}$-space if and only if $M$ is a pseudo simple $R$-module.
\end{theorem}

\begin{proof}
$\Leftarrow:$ Suppose $M$ is a pseudo simple module and choose $[m]\in
\operatorname{EC}(M^{\#})$. It follows that $Rm$ is a simple module. Now,
choose $[m_{1}]\in\overline{\{[m]\}}$ then $m_{1}\neq0$ and $m\mid m_{1}$ by
Proposition \ref{closure}. Hence, $Rm_{1}=Rm$ as $Rm$ is a simple module.
Therefore, $[m]=[m_{1}]$ so $D(M)$ is a $T_{1}$ space.

$\Rightarrow:$ Suppose $D(M)$ is a $T_{1}$-space then $\overline{\{[m]\}}%
=[m]$. Now, choose $a\not \in \operatorname{ann}(m)$ then note that $am\neq0$
and not a generator. Moreover, $[am]\in\overline{\{[m]\}}=\{[m]\}$ then
$Ram=Rm$. Hence, $1-xa\in\operatorname{ann}(m)$ for some $x\in R$ implying
that $\operatorname{ann}(m)$ is a maximal, which completes the proof.
\end{proof}

\section{Compactness, Connectivity and countability axioms of $D(M)$}

\bigskip In this section, we examine compactness, connectivities, and
countability axioms of $D(M).\ $A topological space $X$ is said to be a
\textit{compact space} if every open covering of $X$ has a finite subcover
\cite{Munkres}. In \cite[Proposition 11]{YiKo}, the authors proved that $D(R)$
is never a compact space. In the following, we show that it is possible for $D(M)$ to be a compact space depending on cases.

\begin{example}
\label{excom1}Let $M$ be a finite $R$-module or $EC(M^{\#})$ be a finite set.
Then $D(M)$ is trivially a compact space. In particular, the divisor topology of $\Z$-module $\Z_n$ where $n$ is a composite number, $D(\Z_n)$ is a compact space.
\end{example}

\begin{example}
\label{excom2}Consider $%
\mathbb{Z}
$-module $%
\mathbb{Q}
.\ $Then $EC(%
\mathbb{Q}
^{\#})=\left\{  \left[  \dfrac{a}{b}\right]  :\gcd(a,b)=1\text{ and }0\neq
a,b\in%
\mathbb{Z}
\right\}  $ is an infinite set. Take a finitely many elements $\left[
\dfrac{a_{1}}{b_{1}}\right]  ,\left[  \dfrac{a_{2}}{b_{2}}\right]
,\ldots,\left[  \dfrac{a_{n}}{b_{n}}\right]  \in EC(%
\mathbb{Q}
^{\#})$. Now, choose a prime number $p$ greater than all $a_{i}$'s and $b_{i}%
$'s. Now, put $x=\dfrac{1}{p}\dfrac{a_{1}}{b_{1}}\dfrac{a_{2}}{b_{2}}%
\cdots\dfrac{a_{n}}{b_{n}}.\ $Then note that $x\nmid\dfrac{a_{i}}{b_{i}}$ for
all $i=1,2,\ldots,n.\ $Thus, we have $\left[  x\right]  \notin%
{\displaystyle\bigcup\limits_{i=1}^{n}}
U_{\frac{a_{i}}{b_{i}}}$ and so $EC(%
\mathbb{Q}
^{\#})\neq%
{\displaystyle\bigcup\limits_{i=1}^{n}}
U_{\frac{a_{i}}{b_{i}}}.\ $This ensures that $D(%
\mathbb{Q}
)$ is not a compact space.
\end{example}

\begin{example}
\label{excom3}Consider the $%
\mathbb{Z}
$-module $%
\mathbb{Q}
/%
\mathbb{Z}
\ $and the submodule
\[
E(p)=\left\{  \alpha=\dfrac{a}{p^{n}}+%
\mathbb{Z}
:a\in%
\mathbb{Z}
\text{ and }n\in%
\mathbb{N}
\cup\{0\}\right\}
\]
for a fixed prime number $p.\ $One can verify that $E(p)$ is not cyclic
and all nonzero element $\alpha$ of $E(p)\ $is associated to $\dfrac{1}{p^{n}%
}+%
\mathbb{Z}
$ for some $n\in%
\mathbb{N}
\cup\{0\}.\ $On the other hand, we have $\dfrac{1}{p^{n}}+%
\mathbb{Z}
\ \mid\dfrac{1}{p^{k}}+%
\mathbb{Z}
$\ if and only if $k\leq n$ if and only if $U_{\frac{1}{p^{n}}+%
\mathbb{Z}
}\subseteq U_{\frac{1}{p^{k}}+%
\mathbb{Z}
}.\ $Thus, open sets in $D(E(p))$ are of the form
\[
EC(E(p)^\#)=U_{\frac{1}{p}+%
\mathbb{Z}
}\supseteq U_{\frac{1}{p^{2}}+%
\mathbb{Z}
}\supseteq\cdots\supseteq U_{\frac{1}{p^{n}}+%
\mathbb{Z}
}\supseteq U_{\frac{1}{p^{n+1}}+%
\mathbb{Z}
}\supseteq\cdots
\]
It follows that $D(E(p))$ is a compact space.
\end{example}

\begin{theorem}
\label{tcom}Let $M\ $be an $R$-module. Then $D(M)$ is a compact space if and
only if the family $\left\{  Rm\right\}  _{m\in M^{\#}}$ has only finitely
many minimal elements with respect to inclusion $\subseteq$ if and only if $M$
has only finitely many simple submodules.
\end{theorem}

\begin{proof}
Suppose that $D(M)$ is a compact space. It is clear that $EC(M^{\#})=%
{\displaystyle\bigcup\limits_{m\in M^{\#}}}
U_{m}$. It means that there exists $m_{1},m_{2},\ldots,m_{n}\in M^{\#}$ such that
$EC(M^{\#})=%
{\displaystyle\bigcup\limits_{i=1}^{n}}
U_{m_{i}}.\ $Now, choose an element $Rz\in\left\{  Rm\right\}  _{m\in M^{\#}%
}\ $then $z\in M^{\#}$ and $\left[  z\right]  \in EC(M^{\#})=%
{\displaystyle\bigcup\limits_{i=1}^{n}}
U_{m_{i}}.\ $Hence, there exists $1\leq i\leq n$ such that $\left[  z\right]  \in
U_{m_{i}}$ which implies that $Rm_{i}\subseteq Rz.\ $Thus minimal elements of
the family of $\left\{  Rm\right\}  _{m\in M^{\#}}$ is contained in
$\{Rm_{1},Rm_{2},\ldots,Rm_{n}\}$. For the converse, assume that the family
$\left\{  Rm\right\}  _{m\in M^{\#}}$ has only finitely many minimal elements
with respect to inclusion $\subseteq.$ Now, let $\{Rm_{1},Rm_{2},\ldots
,Rm_{n}\}$\ be the set of minimal elements of $\left\{  Rm\right\}  _{m\in
M^{\#}}.$\ Let $EC(M^{\#})=%
{\displaystyle\bigcup\limits_{j\in\Delta}}
O_{j}$ for some open sets $O_{j}$ in $D(M).\ $Since $\{U_{m}\}_{m\in M^{\#}}$
is a basis for $D(M),\ $we can write $EC(M^{\#})=%
{\displaystyle\bigcup\limits_{j\in\Delta}}
U_{x_{j}}$ for some $x_{j}\in M^{\#}.\ $Note that $\left[  m_{i}\right]  \in
EC(M^{\#})=%
{\displaystyle\bigcup\limits_{j\in\Delta}}
U_{x_{j}}$ for every $1\leq i\leq n$ and so there exists $x_{j_{i}}\in M^{\#}$
such that $m_{i}\ |\ x_{j_{i}}$ which implies that $Rx_{j_{i}}\subseteq
Rm_{i}.\ $By the minimality of $Rm_{i},\ $we have $Rx_{j_{i}}=Rm_{i}$, and so
we conclude that $U_{x_{j_{i}}}=U_{m_{i}}.\ $On the other hand, since
$\{Rm_{1},Rm_{2},\ldots,Rm_{n}\}$ is the set of minimal elements of the family
$\left\{  Rm\right\}  _{m\in M^{\#}},\ $one can similarly show that
$EC(M^{\#})=%
{\displaystyle\bigcup\limits_{j\in\Delta}}
U_{x_{j}}=%
{\displaystyle\bigcup\limits_{i=1}^{n}}
U_{m_{i}}=%
{\displaystyle\bigcup\limits_{i=1}^{n}}
U_{x_{j_{i}}}.\ $Hence, $D(M)$ is a compact space. The rest follows from the
fact that all simple submodules are cyclic and minimal elements of the family
$\left\{  Rm\right\}  _{m\in M^{\#}}.$
\end{proof}

Let $M$ be an $R$-module. A submodule $N$ of $M$ is said to be an
\textit{essential submodule} if $N\cap L=0$ implies $L=0.\ $The socle $Soc(M)$
of an $R$-module $M$ is defined as sum of simple submodules of $M.\ $If there
is no simple submodule, then $Soc(M)=0.\ $An $R$-module $M\ $is said to be a
\textit{finitely cogenerated module} if $%
{\displaystyle\bigcap\limits_{i\in\Delta}}
N_{i}=0$, then there exists a finite index set $\Delta^{\prime}\subseteq
\Delta$ such that $%
{\displaystyle\bigcap\limits_{i\in\Delta^{\prime}}}
N_{i}=0$ for every family $\left\{  N_{i}\right\}  _{i\in\Delta}$ of
submodules of $M\ $\cite{AnFu}. It is well known that $M\ $\ is a finitely
cogenerated if and only if $Soc(M)\ $is essential and finitely generated (See,
\cite[Proposition 10.7]{AnFu}).

\begin{theorem}
\label{tfinitelycog}Let $M$ be an $R$-module. If $D(M)\ $is a compact space,
then $Soc(M)\ $is finitely generated and essential. In this case, $M\ $is a
finitely cogenerated module.
\end{theorem}

\begin{proof}
Suppose that $D(M)\ $is a compact space. Then by the proof of Theorem
\ref{tcom}, there exists finitely many simple submodules $Rm_{1},Rm_{2}%
,\ldots,Rm_{n}$ of $M.\ $This implies that $Soc(M)=%
{\displaystyle\sum\limits_{i=1}^{n}}
Rm_{i}$ and $Soc(M)\ $is finitely generated. On the other hand, by the
proof of Theorem \ref{tcom}, we have $EC(M^{\#})=%
{\displaystyle\bigcup\limits_{i=1}^{n}}
U_{m_{i}}.\ $Now, we will show that $Soc(M)\ $is essential. Let $L\ $be a
submodule of $M\ $such that $L\cap Soc(M)=0.\ $Assume that $L\neq0.\ $In this
case, choose $0\neq x\in L$. If $Rx=M,\ $then we have $Soc(M)=%
{\displaystyle\sum\limits_{i=1}^{n}}
Rm_{i}=0$ which implies that $m_{1},m_{2},\ldots,m_{n}$ are all zero, which is
a contradiction. Thus, we assume that $Rx\neq M,\ $that is, $x\in M^{\#}%
.\ $Since $\left[  x\right]  \in EC(M^{\#})=%
{\displaystyle\bigcup\limits_{i=1}^{n}}
U_{m_{i}},\ $there exists $1\leq j\leq n$ such that $x\ |\ m_{j}$ which
implies that $Rm_{j}\subseteq Rx.\ $As $Rx\cap Rm_{j}=Rm_{j}\subseteq L\cap
Soc(M)=0,\ $we have $m_{j}=0$, again a contradiction. Hence, we have
$L=0,\ $that is, $Soc(M)\ $is essential in $M.\ $The rest follows from
\cite[Proposition 10.7]{AnFu}.
\end{proof}

The converse of the previous theorem is not true in general. See the following example.

\begin{example}
Consider a finite dimensional vector space $V$ over an infinite field $K$ with
$1<\dim_{K}(V).$ For instance,\ let $V=%
\mathbb{R}
^{3}$ and $K=%
\mathbb{R}
.\ $Then $V$ is an Artinian module so it is a finitely cogenerated module. Since
all 1-dimensional subspaces are simple submodules, there are infinitely many
simple submodules. Then by Theorem \ref{tcom}, $D(V)$ is not a compact space.
\end{example}

A topological space$\ X$ is said to \textit{have a countable basis at the
point }$x\in X$ if there exists a countable family $\{U_{k}\}_{k\in%
\mathbb{N}
}$ of neighborhoods of $x$ such that every neighborhood $U$ of $x$ contains at
least one of $U_{k}$ \cite{Munkres}. $X\ $is said to s\textit{atisfy the first
countability axiom} if there exists a countable basis at every point $x\in
X.\ $Also $X\ $is said to \textit{satisfy the second countability axiom} if it
has a countable basis \cite{Munkres}.

\begin{proposition}
\label{countable}(i) $D(M)$ satisfies the first countability axiom.

(ii)\ If $M\ $is an $R$-module with $\operatorname{EC}(M^{\#})$ has countable
element, then $D(M)$ satisfies the second countability axiom.
\end{proposition}

\begin{proof}
$(i):\ $It follows directly from Lemma \ref{smallU_m}.

$(ii):\ $It follows from the fact that $\left\{  U_{m}\right\}  _{m\in M^{\#}%
}$ is the basis of $D(M).$
\end{proof}

A topological space $X\ $is said to be an \textit{ultraconnected space} if the
intersection of two nonempty closed sets is nonempty\ \cite{Richmand}. $X\ $is
said to be a $T_{4}$\textit{-space} if for every two disjoint closed sets $F$
and $K$ of $X,$ there exists two open disjoint sets containing them. It is
clear that every ultraconnected space is trivially a $T_{4}$ space. Also,
$X\ $is called a \textit{normal space} if it is both $T_{1}\ $and $T_{4}$
space \cite{Munkres}. Let $M$ be an $R$-module.\ $M\ $is said to be a
\textit{multiplication module} if its each submodule $N$ has the form $IM$ for
some ideal $I$ of $R$ \cite{Barnard}. It is easy to see that $M$ is a
multiplication module if and only if $N=(N:M)M$ for every submodule $N$ of
$M,\ $where $(N:M)=\{r\in R:rM\subseteq N\}=ann(M/N).$ For more details on
multiplication modules, we refer \cite{AnArTeKo} and \cite{Smith} to the reader.

\begin{proposition}
\label{ultraconnected} Let $M$ be a multiplication $R$-module in which
$ann(M)$ is a prime ideal. Then $D(M)$ is an ultraconnected space. In
particular, $D(M)$ is a $T_{4}$-space.
\end{proposition}

\begin{proof}
Let $K_{1}$ and $K_{2}$ be two nonempty closed sets of $D(M)$. Then there are
elements $[k_{1}]\in K_{1}$ and $[k_{2}]\in K_{2}$ for some $k_{1},k_{2}\in
M^{\#}$. It follows that $(Rk_{1}:M)\neq0$ and $(Rk_{2}:M)\neq0$ since $M$ is
a multiplication module with two nonzero submodules $Rk_{1}$ and $Rk_{2}$.
Then we have $(Rk_{1}:M)\cap(Rk_{2}:M)\supseteq(Rk_{1}:M)(Rk_{2}:M)\neq
0.\ $Since $ann(M)$ is a prime ideal, we can choose $x\in(Rk_{1}:M)\cap
(Rk_{2}:M)-ann(M).\ $Then there exists $m\in M$ such that $0\neq xm\in
Rk_{1}\cap Rk_{2}.\ $Then by Proposition \ref{closure}, $\left[  xm\right]
\in\overline{\{[k_{1}]\}}\cap\overline{\{[k_{2}]\}}\subseteq K_{1}\cap K_{2}$.
Hence, $D(M)$ is ultraconnected space, so is $T_{4}$-space.
\end{proof}

A topological space $X\ $is said to be a \textit{connected space} if it can
not be written as a union of two nonempty disjoint open subsets. Let $a,b\in
X.\ $Then a continuous map $f:[c,d]\rightarrow X$ such that $f(c)=a\ $and
$f(d)=b$ is said to be a \textit{path} from $a$ to $b.\ $If there exists a
path in $X$ for every pair of elements of $X,\ $then $X\ $is said to be a
\textit{path connected space} \cite{Munkres}.\ 

\begin{corollary}
\label{cconnected}If $M\ $is a multiplication $R$-module in which $ann(M)$ is
a prime ideal, then $D(M)$ is a connected and path connected space.
\end{corollary}

\begin{proof}
In view of Proposition \ref{ultraconnected}, we know that $D(M)$ is an
ultraconnected space. The rest follows from the fact that any ultraconnected
spaces are path connected space, so they are connected.
\end{proof}

\section{Further topological properties of $D(M)$}

\bigskip In this section, we examine further topological properties of $D(M)$
such as Hausdorff, $T_{3},T_{5},\ $discrete, and (completely) normal
properties. We know from Example \ref{ex1} and Example \ref{ex2} that $D(M)$
may be Hausdorff/discrete or not. Now, we provide two new examples as follows.

\begin{example}
Consider the $K$-module $K[[x]],$ where $K$ is any field. Then, $K[[X]]^{\#}%
=K[[X]]-\{0\}$ and it is easy to see that every element of $K[[X]]^{\#}$ is
irreducible on $K[[X]]^{\#}$ since $K[[X]]$ is a vector space. Thus,
$D(K[[X]]^{\#})$ is a discrete topology and hence a Hausdorff space.
\end{example}

\begin{example}
Consider ${\mathbb{Z}}$-module ${\mathbb{Z}}_{p^{\alpha}}$ where $p$ is a
prime number and $\alpha\geq3$. Then for any $\overline{x}\in{\mathbb{Z}%
}_{p^{\alpha}}^{\#}$, we can write $\overline{x}=k\overline{p}^{\beta}$ for
some $k\in{\mathbb{Z}}$ with $\gcd(k,p)=1$ and $1\leq\beta\leq\alpha-1$. Now,
we will characterize all irreducibles on ${\mathbb{Z}}_{p^{\alpha}}^{\#}$.
Since ${\mathbb{Z}}.\overline{x}={\mathbb{Z}}.\overline{p}^{\beta}$, we may
assume that $\overline{x}=\overline{p}^{\beta}$. If $\beta\geq2$, then clearly
we have $\overline{p}\mid\overline{p}^{\beta}$ and $\overline{p}^{\beta}%
\nmid\overline{p}$ which implies that $\overline{p}^{\beta}$ is not an
irreducible on ${\mathbb{Z}}_{p^{\alpha}}^{\#}$. Thus, irreducible on
${\mathbb{Z}}_{p^{\alpha}}^{\#}$ are precisely $\{k\overline{p}:\gcd(k,p)=1\}$
which implies that $U_{\overline{x}}=\{[\overline{x}]\}$ if and only if
$[\overline{x}]=[\overline{p}]$. That is, $U_{\overline{p}}=\{[\overline
{p}]\}$ and $U_{\overline{p}^{\beta}}\supsetneq\{[\overline{p}^{\beta}]\}$.
Consequently, $D({\mathbb{Z}}_{p^{\alpha}}^{\#})$ is not a discrete topology
and also not a Hausdorff space. In fact, $D({\mathbb{Z}}_{p^{\alpha}}^{\#})$
is not a $T_{1}$ space since ${\mathbb{Z}}$-module ${\mathbb{Z}}_{p^{\alpha}}%
$, where $\alpha\geq3,\ $is not a pseudo simple ${\mathbb{Z}}$ module by
Proposition \ref{pseudoZn}.
\end{example}

It is well known that the Zariski topology $Spec(R)\ $is a $T_{1}$ space if
and only if it is a Hausdorff space (See, \cite[Proposition 3.5]{Kim}). A
similar argument holds for the divisor topology $D(M).\ $In fact, we prove
that $D(M)$ is a $T_{1}$ space if and only if $D(M)$ is a Hausdorff space if
and only if it is discrete. To prove our main result (Theorem \ref{main}) in
this section, we need the following definition.

\begin{definition}
An $R$-module $M$ is said to satisfy $(\ast)$-condition if $m_{1}$ and $m_{2}$
are two nonzero nongenerator element of $M$ with $Rm_{1}\neq Rm_{2}$ and
whenever $Rm_{1}+Rm_{2}\subseteq Rx$ for some $x\in M$, then $Rx=M$.
\end{definition}

\begin{example}
(i)\ Every vector space trivially satisfies $(\ast)$-condition.

(ii) Every simple module satisfies $(\ast)$-condition.

(iii)$\ {\mathbb{Z}}$-module ${\mathbb{Z}}_{pq}$ is apart from examples $(i)$
and $(ii)$ satisfying $(\ast)$-condition, where $p$ and $q$ are two distinct primes.
\end{example}

Recall from \cite{Tuganbaev} that an $R$-module $M$ is called a \textit{Bezout
module} if each finitely generated submodule of $M$ is cyclic.

\begin{proposition}
Suppose $M$ is a Bezout module satisfying $(*)$-condition. Then, $Rm$ is a
maximal submodule for all nonzero nongenerator $m\in M$. In this case, $M$ is
either a simple module or every factor module $M/N$ is a simple module for all
nonzero submodule $N$ of $M$.
\end{proposition}

\begin{proof}
Let $0\neq m\in M$ be a nongenerator of $M$. We will show that $Rm$ is a
maximal submodule of $M$. Let $Rm\subsetneq N$ for some submodule $N$ of $M$
and choose $m^{\prime}\in N-Rm$. This gives $Rm\neq Rm^{\prime}$. Since $M$ is
a Bezout module, $Rm+Rm^{\prime}$ is a cyclic module so $Rm+Rm^{\prime}=Rx$
for some $x\in M$. Since $M$ satisfies $(\ast)$-condition, we can conclude
that $Rx=M\subseteq N$ which implies $N=M$. Therefore, $Rm$ is a maximal
submodule of $M$. Also since all proper nonzero cyclic submodules are maximal,
then all proper nonzero submodules are maximal. The rest is clear.
\end{proof}

\begin{theorem}
\label{Hausdorff*condition} Let $M$ be an $R$-module, then $D(M)$ is a
Hausdorff space if and only if $M$ satisfies $(\ast)$-condition.
\end{theorem}

\begin{proof}
Suppose that $M$ satisfies $(\ast)$-condition. Let $[m_{1}]\neq\lbrack
m_{2}]\in\operatorname{EC}(M^{\#})$ then $Rm_{1}\neq Rm_{2}$ with two nonzero
nongenerator elements $m_{1}$ and $m_{2}$ of $M$. Now, we will show that
$U_{m_{1}}\cap U_{m_{2}}=\emptyset$. Suppose not, choose $[x]\in U_{m_{1}}\cap
U_{m_{2}}$ for some $x\in M^{\#}$. This gives $x\mid m_{1}$ and $x\mid m_{2}$
which implies that $Rm_{1}+Rm_{2}\subseteq Rx$. Since $M$ satisfies $(\ast
)$-condition, we have $Rx=M$ which is a contradiction. Thus, $U_{m_{1}}$ and
$U_{m_{2}}$ are disjoint two open sets containing $m_{1}$ and $m_{2}$,
respectively. This ensures that $D(M)$ is a Hausdorff space.

For the converse, assume $D(M)$ is a Hausdorff space. Let $m_{1}$ and $m_{2}$
be two nonzero nongenerator of $M$ and $Rm_{1}\neq Rm_{2}$. Let $Rm_{1}%
+Rm_{2}\subseteq Rx$ for some $x\in M$. Assume that $Rx\neq M$ and note that
$[m_{1}]\neq\lbrack m_{2}]\in\operatorname{EC}(M^{\#})$. Since $Rm_{1}%
+Rm_{2}\subseteq Rx$, we conclude that $x\mid m_{1}$ and $x\mid m_{2}$
implying that $[x]\in U_{m_{1}}\cap U_{m_{2}}$, that is, $U_{m_{1}}\cap
U_{m_{2}}\neq\emptyset$. Since $U_{m_{1}}$ and $U_{m_{2}}$ are the smallest
open set containing $m_{1}$ and $m_{2}$, respectively, $D(M)$ is not a
Hausdorff space which leads to a contradiction. Thus, $Rx=M$ and so $M$
satisfies $(\ast)$-property.
\end{proof}

\begin{theorem}
\label{tdiscrete}Let $M$ be an $R$-module, then $D(M)$ is a discrete topology
if and only if every element $m\in M^{\#}$ is an irreducible on $M^{\#}$.
\end{theorem}

\begin{proof}
Let $m\in M^{\#}$ and it is easy to see that $U_{m}=\{[m]\}$ since $m$ is
irreducible on $M^{\#}$. Thus, every singleton $\{[m]\}$ is open set, so
$D(M)$ is a discrete topology. Conversely, let $D(M)$ be a discrete topology.
Pick $m\in M^{\#}$ and $m_{1}\in M^{\#}$ such that $m_{1}\mid m$. Since $D(M)$
is a discrete space, $\{[m]\}$ is an open set. Note that $U_{m}$ is the
smallest open set containing $[m]$. Thus, $U_{m}=\{[m]\}$. Since $m_{1}\mid
m$, we have $[m_{1}]\in U_{m}=\{[m]\}$ which implies that $[m_{1}]=[m]$, that
is, $m\mid m_{1}$. Hence, $m$ is an irreducible on $M^{\#}$.
\end{proof}

\begin{theorem}
\label{main}Let $M$ be an $R$-module. Then the following statements are equivalent.

(i)\ $M$ is a pseudo simple module.

(ii)\ $M\ $satisfies $(\ast)$-condition.

(iii)\ $D(M)\ $is a discrete space.

(iv)\ $D(M)$ is a Hausdorff space.

(v)\ $D(M)$ is a $T_{1}$-space.
\end{theorem}

\begin{proof}
$(i)\Leftrightarrow(v):$\ Follows from Theorem \ref{T_1-p.simple}.

$(ii)\Leftrightarrow(iv):\ $Follows from Theorem \ref{Hausdorff*condition}.

$(iii)\Rightarrow(iv)\Rightarrow(v):\ $Clear.

$(i)\Leftrightarrow(v)\Rightarrow(iii):\ $Suppose that $M$ is a pseudo simple
module and then by Theorem \ref{tdiscrete}, we only need to show that every
element $m\in M^{\#}$ is an irreducible on $M^{\#}$. Let $m_{1}\ |\ m$ for
some $m,m_{1}\in M^{\#}.\ $Since $M\ $is a pseudo simple module, $Rm_{1}$ is
simple, and thus we have $Rm=Rm_{1}$ which implies that $m\ |\ m_{1}.\ $Thus,
$m$ is an irreducible on $M^{\#}.\ $
\end{proof}

In a topological space $X,\ $two subsets $A$ and $B$ are said to be
\textit{seperated in} $X\ $if $A\cap\overline{B}=\emptyset$ and $\overline
{A}\cap B=\emptyset.\ $A topological space $X\ $is called a $T_{5}$\textit{
space} if for every two seperated sets $A$ and $B,\ $there exist two open
neighborhoods containing them \cite{Richmand}.\ $X\ $is said to be a
\textit{completely normal} if it is both a $T_{1}\ $and $T_{5}$ space
\cite{Richmand}.

\begin{proposition}
\label{T5}Let $M$ be an $R$-module. The following statements are satisfied.

(i)\ If $M\ $is a uniserial module, then $D(M)$ is a $T_{5}$-space.

(ii)$\ $If $M\ $is a torsion free multiplication nonuniserial module, then
$D(M)$ is not a $T_{5}$-space.
\end{proposition}

\begin{proof}
$(i):\ $Let $M$ be a uniserial module. Choose $m_{1},m_{2}\in M^{\#}.\ $Then
by Lemma \ref{divide}, $m_{1}\ |\ m_{2}$ or $m_{2}\ |\ m_{1}.$\ This ensures
that $\left\{  \left[  m_{1}\right]  \right\}  $ and $\left\{  \left[
m_{2}\right]  \right\}  $ are not seperated in $D(M)$ by Proposition
\ref{closure}. Thus, $D(M)$ is trivially a $T_{5}$-space.

$(ii):\ $Suppose that $M\ $is a torsion free multiplication nonuniserial
module. Then there exist $m_{1},m_{2}\in M^{\#}$ such that $m_{1}\nmid
m_{2}\ $and $m_{2}\nmid m_{1}.$ Choose $m\in M^{\#}$ then note that
$(Rm:M)\neq0$ and take $0\neq x\in(Rm:M).\ $This gives that $xm_{1}\nmid
xm_{2}\ $and $xm_{2}\nmid xm_{1}$ which implies that $\left\{  \left[
xm_{1}\right]  \right\}  $ and $\left\{  \left[  xm_{2}\right]  \right\}  $
are separated sets. Let $\left[  m^{\prime}\right]  \in U_{m}$ then we have
$m^{\prime}\ |\ m.\ $Since $x\in(Rm:M),\ $we get $xm_{1},xm_{2}\in Rm,$ that
is, $m\ |\ xm_{1}\ $and $m\ |\ xm_{2}.$\ Since $m^{\prime}\ |\ m,\ $we
conclude that $m^{\prime}\ |\ xm_{1}\ $and $m^{\prime}\ |\ xm_{2},\ $that is,
$\left[  m^{\prime}\right]  \in U_{xm_{1}}\cap U_{xm_{2}}$ and so
$U_{m}\subseteq U_{xm_{1}}\cap U_{xm_{2}}.\ $Hence, $D(M)$ is not a $T_{5}$-space.
\end{proof}

\begin{corollary}
\label{ccompletely}Let $M$ be a torsion free multiplication $R$-module. If $M$
is a nonuniserial module, then $D(M)$ is not a completely normal space.
\end{corollary}

\begin{proof}
It follows from Proposition \ref{T5} (ii).
\end{proof}

A topological space $X$ is said to be a $T_{3}$\textit{-space} if for every
closed set $F$ and $x\in X-F,\ $there exist two open disjoint sets $U\ $and
$V$ containing them \cite{Munkres}.

\begin{remark}
\label{rcompletely}We note here that since the discrete space satisfies each
of the separation axioms, $M$ is a pseudo simple $R$-module implies that
$D(M)$ is a $T_{i}$ space for $1\leq i\leq5$, and also $D(M)$ is a
(completely) normal space if and only if $M$ is a pseudo simple $R$-module.
\end{remark}


\begin{thebibliography}{99}                                                                                               %


\bibitem {Hamed}Ahmed, H. (2018). S-Noetherian spectrum condition.
Communications in Algebra, 46(8), 3314-3321.

\bibitem {AnArTeKo}Anderson, D. D., Arabaci, T., Tekir, \"{U}., \& Ko\c{c}, S.
(2020). On S-multiplication modules. Communications in Algebra, 48(8), 3398-3407.

\bibitem {AnFu}Anderson, F. W., \& Fuller, K. R. (1992). Rings and categories
of modules (Vol. 13). Springer Science \& Business Media.

\bibitem {Anderson}Anderson, D. D., \& Winders, M. (2009). Idealization of a
module. Journal of commutative algebra, 1(1), 3-56.

\bibitem {Arenas}Arenas, F. G. (1999). Alexandroff spaces. Acta Math. Univ.
Comenianae, 68(1), 17-25.

\bibitem {AtMac}Atiyah, M.F., Macdonald, I.G. (1969). Introduction to
Commutative Algebra. Addison-Wesley Publishing Co., Reading, Mass., London-Don
Mills, Ontario.

\bibitem {Barnard}Barnard, A. (1981). Multiplication modules. J. Algebra.,
71(1), 174-178.

\bibitem {Callialp}Callialp, F., Ulucak, G., \& Tekir, \"{U}. (2017). On the
Zariski topology over an $L$-module $M$. Turkish Journal of
Mathematics, 41(2), 326-336.

\bibitem {Ceken}\c{C}eken, S. (2023, December). Dual Zariski Spaces of
Modules. In Algebra Colloquium (Vol. 30, No. 04, pp. 569-584). World
Scientific Publishing Company.

\bibitem {Ceken2}\c{C}eken, S. (2022). On S-second spectrum of a module.
Revista de la Real Academia de Ciencias Exactas, F\'{\i}sicas y Naturales.
Serie A. Matem\'{a}ticas, 116(4), 171.

\bibitem {Smith}El-Bast, Z. A., \& Smith, P. F. (1988). Multiplication
modules. Communications in Algebra, 16(4), 755-779.

\bibitem {FacSal}Facchini, A., \& Salce, L. (1990). Uniserial modules: sums
and isomorphisms of subquotients. Communications in Algebra, 18(2), 499-517.

\bibitem {Hars}Hartshorne, R. (2013). Algebraic Geometry, vol. 52. Springer.

\bibitem {Lu}Lu, C. P. (1977). Factorial modules. The Rocky Mountain Journal
of Mathematics, 7(1), 125-139.

\bibitem {Kim}Kim, C. (2019). The Zariski topology on the prime spectrum of a
commutative ring (Doctoral dissertation, California State University, Northridge).

\bibitem {Munkres}Munkres, J. (2000). Topology, 2nd edn. Prentice Hall, Upper
Saddle River.

\bibitem {Moghaderi}Moghaderi, J., \& Nekooei, R. (2010). Valuation, discrete
valuation and Dedekind modules. International Electronic Journal of Algebra,
8(8), 18-29.

\bibitem {Ohm}Ohm, J. and Pendleton,\ R.\ (1968). Rings with Noetherian
spectrum, Duke Math. J. 35, 631--639.

\bibitem {OzNaTeKo}\"{O}zen, M., Naji, O. A., Tekir, \"{U}., \& Ko\c{c}, S.
(2023). On Modules Satisfying S-Noetherian Spectrum Condition. Communications
in Mathematics and Statistics, 11(3), 649-662.

\bibitem {Richmand}Richmond, T. (2020). General Topology: An Introduction. De Gruyter.

\bibitem {Sharp}Sharp, R. Y. (2000). Steps in commutative algebra (No. 51).
Cambridge university press.

\bibitem {Steen}Steen, L. A., Seebach, J. A., \& Steen, L. A. (1978).
Counterexamples in topology (Vol. 18). New York: Springer.

\bibitem {Tuganbaev}Tuganbaev, A. A. (2009). Bezout modules and rings. Journal
of Mathematical Sciences, 163(5), 596-597.

\bibitem {Yildiz}Y\i ld\i z, E., Ersoy, B. A., Tekir, \"{U}., \& Ko\c{c}, S.
(2021). On S-Zariski topology. Communications in Algebra, 49(3), 1212-1224.

\bibitem {YiKo}Yi\u{g}it, U., \& Ko\c{c}, S. (2025). On divisor topology of
commutative rings. Ricerche di Matematica, 1-13.
\end{thebibliography}
\end{document}